\documentclass{article}

\usepackage{arxiv}

\usepackage[utf8]{inputenc} 
\usepackage[T1]{fontenc}    
\usepackage[unicode]{hyperref}       
\usepackage{url}            
\usepackage{booktabs}       
\usepackage{amsfonts}       
\usepackage{nicefrac}       
\usepackage{lipsum}
\usepackage{amsmath,amsthm,amssymb} 
\usepackage{graphicx}
\usepackage{setspace}
\usepackage[normalem]{ulem}     
\usepackage{cite}    

\usepackage{tikz}

\usetikzlibrary{arrows,decorations.pathmorphing,backgrounds,positioning,fit,petri}

\makeatletter
\newcommand{\ruleset}[1]{\textsc{#1}}
\newcommand{\CN}[2]{\ruleset{CN\ensuremath{(#1,#2)}}}

\newtheorem{theorem}{Theorem}[section]
\newtheorem{proposition}[theorem]{Proposition}

\newtheorem{lemma}[theorem]{Lemma}
\newtheorem{definition}[theorem]{Definition}
\newtheorem{remark}[theorem]{Remark}

\newcommand{\tref}[1]{Theorem~\textup{\ref{#1}}}

\newcommand{\cref}[1]{Corollary~\textup{\ref{#1}}}

\newcommand{\fref}[1]{Figure~\textup{\ref{#1}}}

\def\p{\ensuremath{{\boldsymbol p}}}

\def\P{\ensuremath{\mathcal{P}}}
\def\N{\ensuremath{\mathcal{N}}}

\def\O{\ensuremath{\mathbf{0}}}
\def\x{\ensuremath{\mathbf{\underline{x}}}}
\def\x13{\ensuremath{x_1x_2x_3}}
\def\xx14{\ensuremath{x_1x_2x_3x_4}}
\def\xp13{\ensuremath{x_1'x_2x_3}}

\usepackage{xcolor}

\usepackage{lscape,pdflscape}
\usepackage{multirow}
\usepackage{subcaption}
\usepackage{caption}

\usepackage{array}
\usepackage{tabulary}
\newcolumntype{K}[1]{>{\centering\arraybackslash}p{#1}}

\title{Circular Nim Games \CN{7}{4}}

\author{
  Matthieu Dufour\\
  Département de Mathématiques\\
  Université du Québec à Montréal, Canada \\  \texttt{dufour.matthieu@uqam.ca} \\
   \And
 Silvia Heubach \\
  Department of Mathematics\\
  California State University, Los Angeles, CA 90032 \\  \texttt{sheubac@calstatela.edu}
    \And 
  Anh Vo \\
  \texttt{anhvo1979@gmail.com}}

\begin{document}
\maketitle
\onehalfspacing
\begin{abstract}

Circular Nim is a two-player impartial combinatorial game consisting of $n$ stacks of tokens placed in a circle. A move consists of choosing $k$ consecutive stacks and taking at least one token from one or more of the stacks. The last player able to make a move wins. The question of interest is: Who can win from a given position if both players play optimally? In an impartial combinatorial game, there are only two types of positions. An \N-position is one from which the next player to move has a winning strategy. A \P-position is one from which the next player is bound to lose, no matter what moves s/he makes. 
Therefore, the question who wins is answered by identifying the \P-positions. 
We will prove results on the structure of the \P-positions for $n = 7$ and $k = 4$, extending known results for other games in this family. The interesting feature of the set of \P-positions of this game is that it splits into different subsets, unlike the structure for the known games in this family. 
\end{abstract}


\section{Introduction}

The game of Nim has been played since ancient times, and the earliest European references to Nim are from the beginning of the sixteenth century. Its current name was coined by Charles L. Bouton of Harvard University, who also developed the complete theory of the game in 1902 \cite{Bouton}. Nim plays a central role among impartial games as any such game is equivalent to a Nim stack \cite{WinningWays}. Many variations and generalizations of Nim have been analyzed. They include subtraction games, Wythoff’s game, Nim on graphs and on simplicial complexes, Take-away games, Fibonacci Nim, etc. \cite{AlbNowWol,rEeS96,tsF98,tsF,dG74,rfG74,dH2010,ehM10,ajS,rS37,mjW63,waW07}. We will study a particular case of another variation, called  Circular Nim, which was introduced in \cite{mDsH09}.

\begin{definition}
In Circular Nim, $n$ stacks of tokens are arranged in a circle. A move consists of choosing $k$ consecutive stacks and then removing at least one token from at least one of the $k$ stacks. The last player who is able to make a legal move wins. We denote this game by $\CN{n}{k}$.
\end{definition}

Circular Nim is an example of a combinatorial game, in which the two players alternately move. There is a set, usually finite, of possible \emph{positions} of the game. The rules of the game specify for both players and each position the legal moves to other positions, which are called \emph{options}. We say a position in a game is a \emph{terminal position} if no moves are possible from it. If the rules make no distinction between the players, that is, both players have the same options to move to, then the game is called \emph{impartial}; otherwise, the game is called \emph{partisan}. The game ends when a terminal position is reached.
Under the \emph{normal-play rule}, the last player to move wins. Otherwise, under the \emph{misère-play rule}, the last player to move loses. More background on combinatorial games can be found in \cite{AlbNowWol,WinningWays,tsF}.

Since we have complete knowledge of the game, the players are assumed to play optimally. Thus, we can study the question: \textbf{“Which player will win the game when playing from a given position?”} Impartial games are easier to analyze than partisan games as they have only two types of positions (= outcomes classes) \cite{tsF}. The outcome classes are described from the standpoint of which player will win when playing from the given position. An \N-position indicates that the \textbf{N}ext player to play from the current position can win, while a \P-position indicates that the \textbf{P}revious player, the one who made the move to the current position, is the one to win. Thus, the current player is bound to lose from this position, no matter what moves she or he makes. A winning strategy for a player in an $\N$-position is to move to one of the $\P$-positions.

\begin{definition} 
In a Circular Nim game, a \emph{position} is represented by the vector $\p = (p_1,p_2, \dots{},p_n)$ of non-negative entries indicating the heights of the stacks in order around the circle. We denote an option of $\p$ by $\p' =(p'_1,p'_2,\dots{},p'_n)$, and use the notation $\p \rightarrow \p'$ to denote a legal move from $\p$ to $\p'$.
\end{definition}

Note that a position in Circular Nim is determined only up to rotational symmetry and reflection (reading the position forward or backward). The only terminal position of $\CN{n}{k}$ is $\O:=(0, 0, \dots{} , 0)$, for all $n$ and $k$. In addition, we do not have to play on all $k$ stacks that are selected. 

\fref{fig:Label8-CN(7,4)} shows an example of the position $\p = (1, 7, 5, 6, 2, 3, 6) \in \CN{7}{4}$ and one possible move, to option $\p' = (0, 1, 5, 4, 2, 3, 6)$, where the four stacks enclosed by squares are the stacks that were selected for play. Note that no tokens were taken from the stack of height 5. 

\begin{figure}[htbp]
\begin{center}
\includegraphics[scale=0.60]{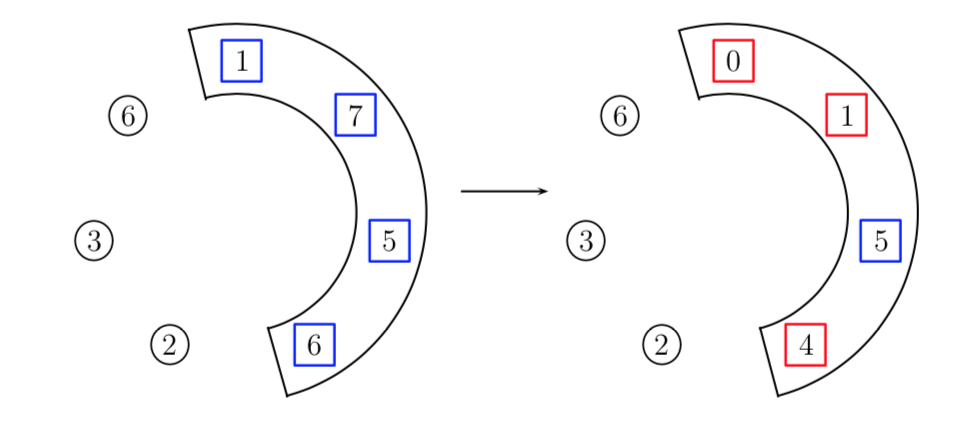}
  \caption{A move from $\p = (1,7,5,6,2,3,6)$ to $\p' = (0,1,5,4,2,3,6)$.}
  \label{fig:Label8-CN(7,4)} 
  \end{center}
\end{figure}

Dufour and Heubach \cite{mDsH09} proved general results on the set of \P-positions  of \CN{n}{1}, \CN{n}{n}, and \CN{n}{n-1} for all $n$. These general cases cover all games for $n \leq 3$. They also gave results for all games with $n \leq 6$ except for \CN{6}{2}, and also solved the game \CN{8}{6}. In this paper, the main  result is on the \P-positions for 
\CN{7}{4}. One sign of the increase in complexity as $n$ and $k$ increase is that, unlike the results for the cases already proved, we no longer can describe the set of \P-positions as a single set, which makes the proofs more complicated.  

To prove our main result, we use the following theorem.

\begin{theorem} [Theorem 1.2, \cite{tsF}] \label{thm:howtoprove}
Suppose the positions of a finite impartial game can be partitioned into mutually exclusive sets A and B with these properties: 
\begin{enumerate}
    \item[\textnormal {I.}] Every option of a position in A is in B; 
    \item[\textnormal {II.}] Every position in B has at least one option in A; and 
    \item[\textnormal {III.}] The terminal positions are in A.
\end{enumerate}
Then A is the unique set of \P-positions and B is the unique set of \N-positions. 
\label{thm: 05}
\end{theorem} 

We use \tref{thm: 05} to show that the conjectured set of \P-positions satisfies the properties of set A and its complement. Property (III) is the easiest one to show, while Property (II) is usually the most difficult part to prove because one has to find a legal move from every \N-position to some \P-position. We are now ready to start our analysis of \CN{7}{4}.

\section{The Game \CN{7}{4}}
\label{sec:CN74}

 In the discussion of \CN{7}{4}, we will use the  generic position $\p = (a, b, c, d, e, f, g)$.  Since positions of \CN{7}{4} are only determined up to rotation and reflection (reading clock-wise or counter clock-wise), we will assume that in a generic position $a$ is a minimum.  Figure~\ref{fig:Pos-CN(7,4)} shows a generic position $(a, b, c, d, e, f, g)$  where  the minimum stack in rendered in red (gray). Note that to avoid cumbersome notation, we will use the label, say $a$, to refer to either the stack itself or to its number of tokens. Which one it is will be clear from the context.

\begin{figure}[htbp]
  \centering
  \includegraphics[scale=0.3]{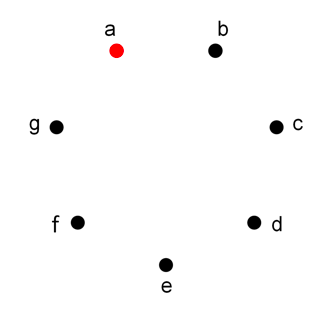}
  \caption{A generic position in the game \CN{7}{4}, with $a = \min(\p)$.}
  \label{fig:Pos-CN(7,4)} 
\end{figure}

Here is our main result,  with a visualization of the \P-positions of \CN{7}{4} given in \fref{fig: S}.  \\

\begin{theorem}\label{thm:CN(7,4)}
Let $\p=(a,b,c,d,e,f,g)$ with $a = \min(\p)$.
The \P-positions of \CN{7}{4} are given by $S$ = $S_1 \cup S_{2} \cup S_3 \cup S_{4}$, where: 

\begin{itemize}
    \item $S_1 = \{\p \mid a = b = 0, \> c = g>0,\> d + e + f = c\}.$
    \item $S_2 = \{\p \mid \p=(a,a,a,a,a,a,a)\}.$
    \item $S_3 = \{\p \mid a = b, \>c = g, \>d = f,  \> a + c = d + e, \> \> 0< a < e\},$ and
    \item $S_4 = \{\p \mid a = f, \> b + c = d + e = g + a, \>a < \min\{b,\>e,\},a < \max\{c,d\}\}.$
\end{itemize}
\label{thm: 06}
\end{theorem}

\begin{figure}[ht]
     \centering
     \begin{subfigure}[b]{0.22\textwidth}
         \centering
         \includegraphics[width=.8\textwidth]{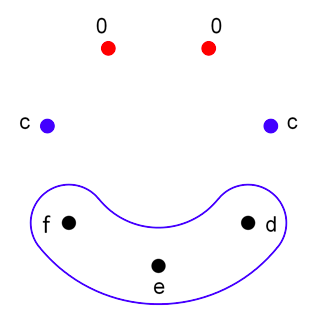}
         \caption{$\P$-positions in $S_1$}
         \label{fig: S1}
     \end{subfigure}
     \hfill
     \begin{subfigure}[b]{0.22\textwidth}
         \centering
         \includegraphics[width=0.8\textwidth]{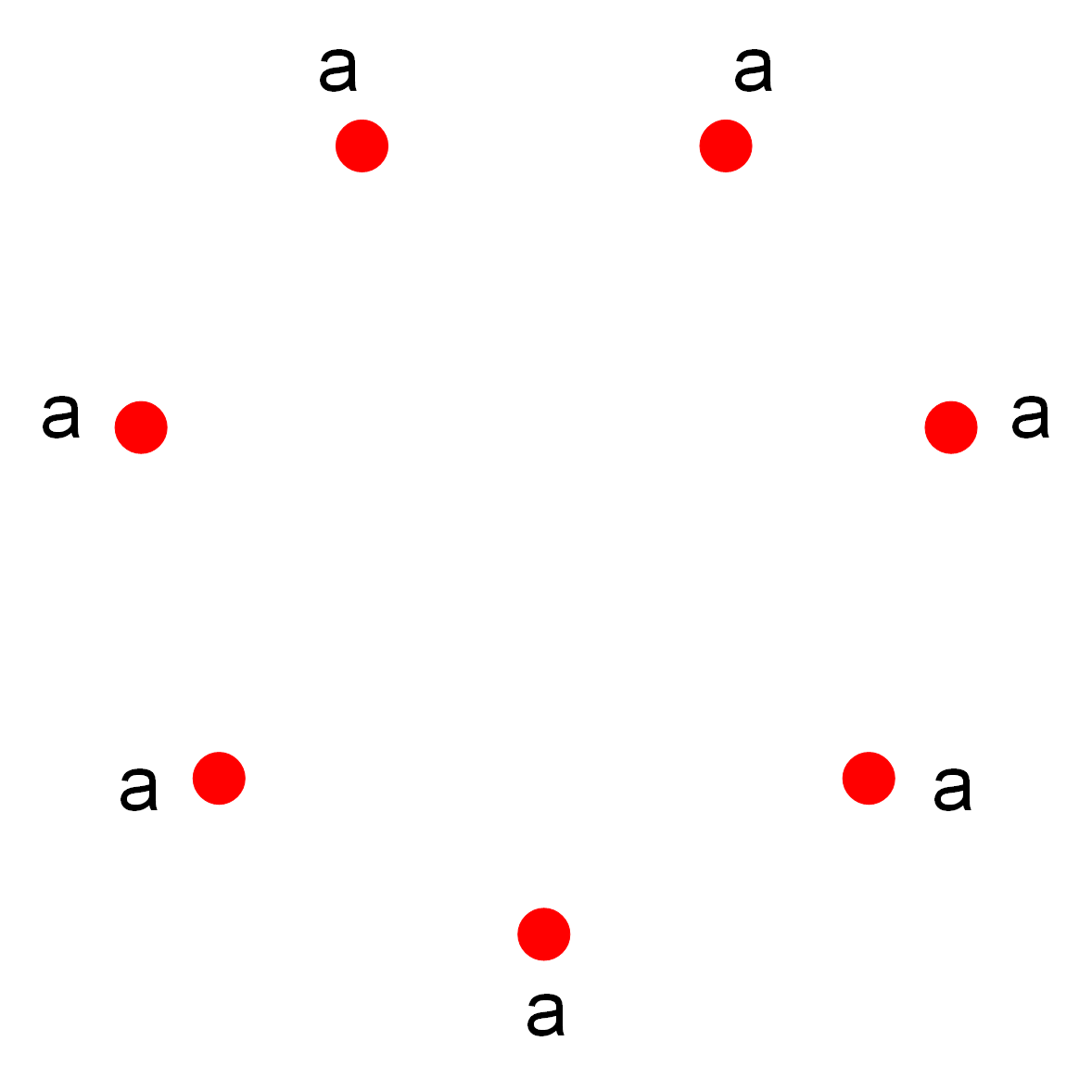}
         \caption{$\P$-positions in $S_2$}
         \label{fig: S2}
     \end{subfigure}
     \hfill
     \begin{subfigure}[b]{0.22\textwidth}
         \centering
         \includegraphics[width=0.8\textwidth]{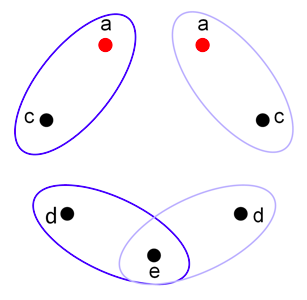}
         \caption{$\P$-positions in $S_3$}
         \label{fig: S3}
     \end{subfigure}
     \hfill
     \begin{subfigure}[b]{0.22\textwidth}
         \centering
         \includegraphics[width=0.8\textwidth]{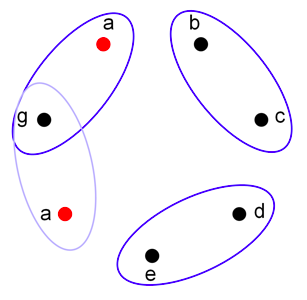}
         \caption{$\P$-positions in $S_4$}
         \label{fig: S4}
     \end{subfigure}

        \caption{Visualization of the \P-positions of \CN{7}{4}. The sums of groups of stacks that are encircled are equal to each other or equal to the blue stack heights.}
        \label{fig: S}
\end{figure}

Note that all the subsets of $S$ are disjoint. The condition $a < \max\{c,d\}$ of $S_4$ prohibits a pair of adjacent minima, which all other sets have. Also, $S_2$ is disjoint from the other sets since they all have a strict inequality  condition. Finally,   $ S_1 \cap S_3 = \varnothing$ since $a>0$ for $ S_3$.
\\

Condition (III) of Theorem~\ref{thm:howtoprove} is satisfied because the only terminal position is $\O \in S_2$. We deal with the other two conditions in the two subsections below. The following  definitions and remarks will aid us in the proofs of our results. Note that we assume $a$ to be the minimum, not necessarily unique. In the proofs, we will denote the minimal and maximal values of a target position by $m$ and $M$, respectively. 

The positions in $S$ have specific geometric features which we will name to make the proofs easier to read. 

\begin{definition}
A \emph{tub configuration} $xmmx$ is a set of four adjacent stacks that consists of a pair of adjacent minima (of the position) surrounded by two stacks of equal height. There are three other stacks in the position, which we denote by $\x13$ unless we know the actual stack heights.  The opposite of a tub configuration is a \emph{peak} $xXx$, a set of three adjacent stacks with $x <X$. If $x$ and $X$ are the  minimum and the maximum, respectively, of the position,  then we call this configuration a \emph{minmax peak}. A position with a peak contains four other stacks  which we denote by $\xx14$ unless we know the actual stack heights. Finally, there is the \emph{common sum requirement}, in which pairs of consecutive stacks have to have the same sum, with one \emph{overlap stack} contributing to two sums.
\end{definition} With these definitions, we can make the following remarks regarding the specific features of each subset of $S$.

\begin{remark}\label{rem:props} \hfill
\begin{enumerate} 
\item \label{rem:props1}  In $S_3$, $a<e$ and the sum conditions imply that $c>\max\{a,d\}$ and $c \ge e$. 
\item \label{rem:S4vals} In $S_4$, we have the following inequalities: $a < \min\{b,e\}$ implies that $g> \max\{c,d\}$ due to the common sum requirement. Furthermore, $g>a$.
\item \label{rem:tubS1S3} Positions in $S_1\cup S_3$ contain a tub configuration. We have 
\begin{itemize}
    \item $\p \in S_1$ needs to satisfy the \emph{tri-sum condition}:  $x_1+x_2+x_3 = x$
    \item $\p \in S_3$ needs to satisfy: $x_1=x_3$ and $x_2+x_3 = a+x$.
\end{itemize}
\item \label{rem:tubcreate}  When trying to move to $\p' \in S_1 \cup S_3$, we can  always create a tub configuration with a new smaller minimum $m'<a$ by playing on three adjacent stacks as follows:  Create a pair of stacks whose common height is a new minimum $m' < a =\min(\p)$. Reduce the larger of the two stacks adjacent to the pair to the height of the smaller. This height gives the value of $x$ in the tub configuration $xm'm'x$.  Any remaining play has to occur on $x_1$, the stack adjacent to the stack that was decreased to $x$. In labeling the remaining three stacks, we are reading the position  starting from the minima in the direction of the stack whose height was reduced to $x$. 
Note that we cannot play on $x_2$ and $x_3$, so for $S_1$, the tri-sum $x_1+x_2+x_3 \ge x_2+x_3$, and for $S_4$, the sum  $x_2+x_3$ cannot be adjusted. 
\item \label{rem:peak}  Positions in $ S_4$ always contain a minmax peak, while positions in $S_3$ may contain a peak. In either case, the remaining four stacks have to satisfy  that $x_1+x_2=x_3+x_4=x+X$. 
 \item  \label{rem:S4} Positions in $S_4$ have either two or  three minima. If  $c=a=m$,  then $\p = (m,M,m,d,e,m,M)$, that is, two  maxima alternate with three minima. Otherwise, the two minima are separated by the maximum.
 \item \label{rem:comsum} Positions in $S_3 \cup S_4$ have the {\em common sum requirement}. Positions in $S_2$ automatically satisfy the common sum requirement. It is relatively easy to see that if we keep the \underline{same} overlap stack, then play on any 4 consecutive stacks from a position $\p$ with common sums leaves at least one sum unchanged, while at least one other sum is decreased, so the common sum requirement cannot be satisfied in $\p'$. Specifically, there is no move from $S_2$ to $S_3 \cup S_4$  since any stack is an overlap stack in $S_2$.
\end{enumerate}
\
\end{remark}

We are now ready to embark on the proofs.

\subsection{There is no move from \texorpdfstring{$\p \in S$ to $\p' \in S$}{\bf{p} in S to \bf{p}' in S}}\label{sub:PnoP}

\begin{proposition} 
If $\p \in S$, then $\p' \notin S$.
\label{prop:noloselose}
\end{proposition}

 \begin{proof}
 To prove condition (I) of Theorem~\ref{thm:howtoprove} we will use the equivalent statement that there is no move from a $\P$-position to another $\P$-position. For each of the four subsets of $S$, we consider moves to all the other sets. 
 
{ \bf Moves from $S_1$:} We start with $\p = (0,c,d,e,f,c,0)\in S_1$, with $d+e+f=c$.  Note that we cannot move to $\p' \in S_1 \cup S_2$ because in either case, we would have to play on the five stacks $c d e \!f \!c$ to simultaneously reduce the $c$ stacks and the sum to a new value $c' < c$ in the case of $S_1$ and $c'=0$ in the case of $S_2$. A move to $S_3$ is not possible since the minimum in $S_3$ is bigger than zero.   A move to  $\p' \in S_4$ is not possible since $S_4$ does not have adjacent minima by  Remark~\ref{rem:props}(\ref{rem:S4}).  Thus, no move is possible from $S_1$ to $S$.

 { \bf Moves from $S_2$:} Now assume that $\p=(a,a,a,a,a,a,a) \in S_2$ with $a>0$  because $\p$ is the terminal position for $a=0$.  To move to $S_1$, we have to create a tub configuration of the form $x00x$, which requires play on at least three stacks. We can at most reduce one of the three remaining stacks  $x_1x_2x_3=aaa$, so the sum $x_1+x_2+x_3\ge 2a$, while $x=a$, so there is no move from $S_2$ to $S_1$. Clearly, one cannot move from $S_2$ to $S_2$. By Remark~\ref{rem:props}(\ref{rem:comsum}) there is no move from $S_2$ to $S_3\cup S_4$.

{ \bf Moves from $S_3$:} Let $\p=(a,a,c,d,e,d,c) \in S_3$.  To move to $S_1 \cup S_3$, we have to create a tub configuration of the form $xa'a'x$, with $a'=0$ for $\p \in S_1$ and $a'\le a$ for $\p \in S_3$.  First we consider play when the minima $a'$ of $\p'$ are located at the $a$ stacks. For a move to $S_1$, we play on both $a$ stacks making them zero, and then either  reduce both $c$ stacks or one of the $d$ stacks, but not both. In either case, we have that $x \le c$  and the tri-sum  $d'+e+d \ge d+e=a+c> c$, so the tri-sum condition is not satisfied. For a move to $S_3$, the overlap stack remains at the same location, and by Remark~\ref{rem:props}(\ref{rem:comsum}), there is no move to $S_3$. 

Now we look at the cases where we create a tub configuration $xa'a'x$ elsewhere.  In each case, we use play on three stacks as described in Remark~\ref{rem:props}(\ref{rem:tubcreate}). By symmetry of positions in $S_3$ we have to consider the three possibilities indicated in Figure~\ref{fig:S3-moves}.  They are $x=a$ with $\x13=edc$, $x=a$ with $\x13=dca$, or $x=d$ with $\x13=aac$ (since $c>d$ by Remark~\ref{rem:props}(\ref{rem:props1}), so we read counter-clockwise). By Remark~\ref{rem:props}(\ref{rem:tubS1S3}), we need to satisfy the conditions $x_1+x_2+x_3 = x+0 = x+a'$ for $\p \in S_1$ and both $x_1=x_3$ and $x_2+x_3 = x+a'$ for $\p \in S_3$. We will show that even if we reduce $x_1$ to zero, we will not be able to satisfy the respective sum conditions. 
When $x = a$, then $x_2+x_3 \ge \min\{d+c,c+a\}>a+a'=x+a'$, and for $x=d$, $a+c=d+e>d+a'=x+a'$. Thus, $\p' \notin S_1\cup S_3$. It is also not possible to move to  $\p' \in S_2$, since by Remark~\ref{rem:props}(\ref{rem:props1}), $\min\{c,e\}>a$, so we would need to play on five stacks to reduce   $cdedc$ to $aaaaa$. 

\begin{figure}[ht]
     \centering
     \begin{subfigure}[b]{0.35\textwidth}
         \centering
         \includegraphics[width=0.8\textwidth]{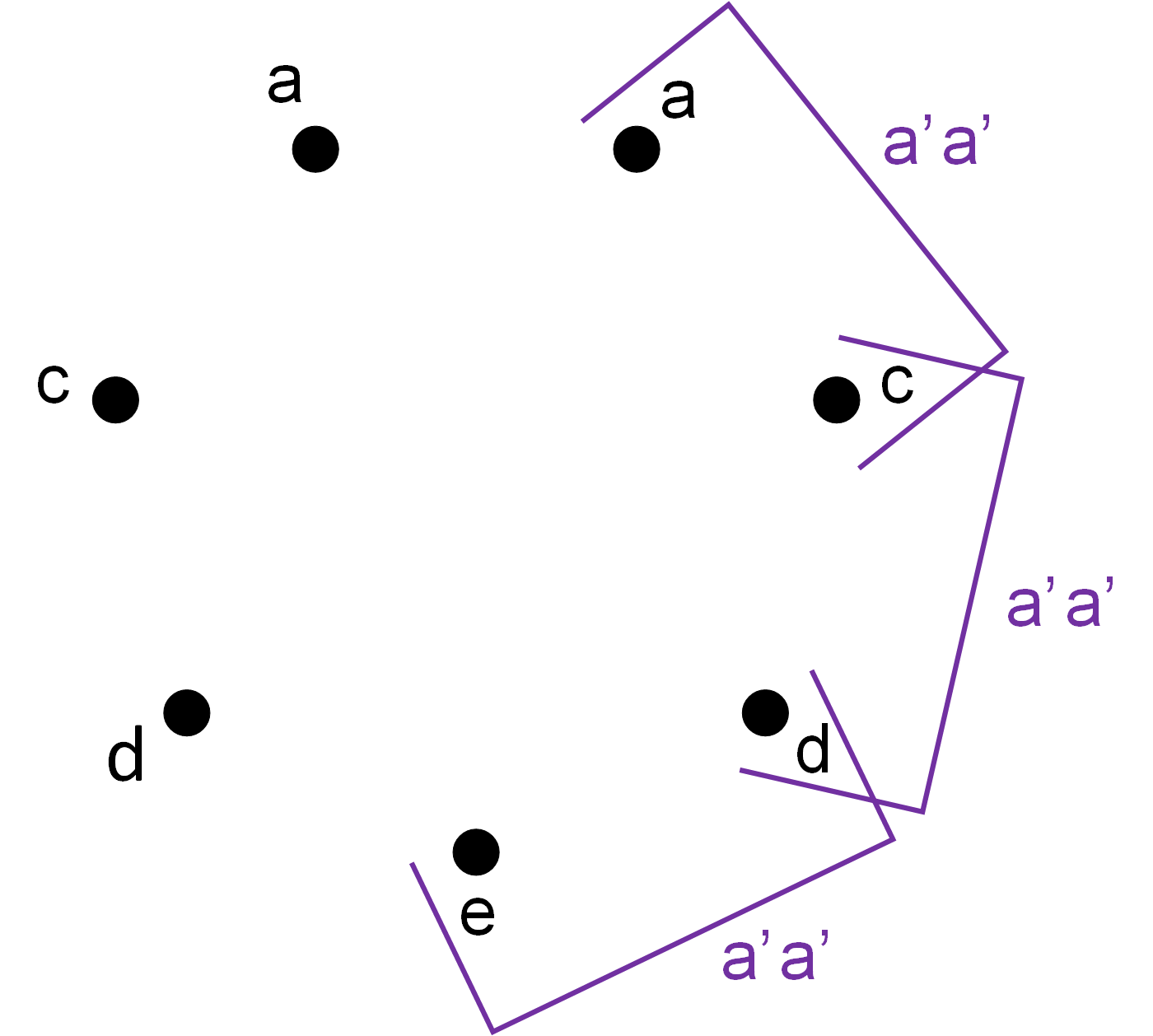}
         \caption{}
         \label{fig:S3-moves}
     \end{subfigure}
     \hspace{0.05\textwidth}
     \begin{subfigure}[b]{0.35 \textwidth}
         \centering
         \includegraphics[width=0.8\textwidth]{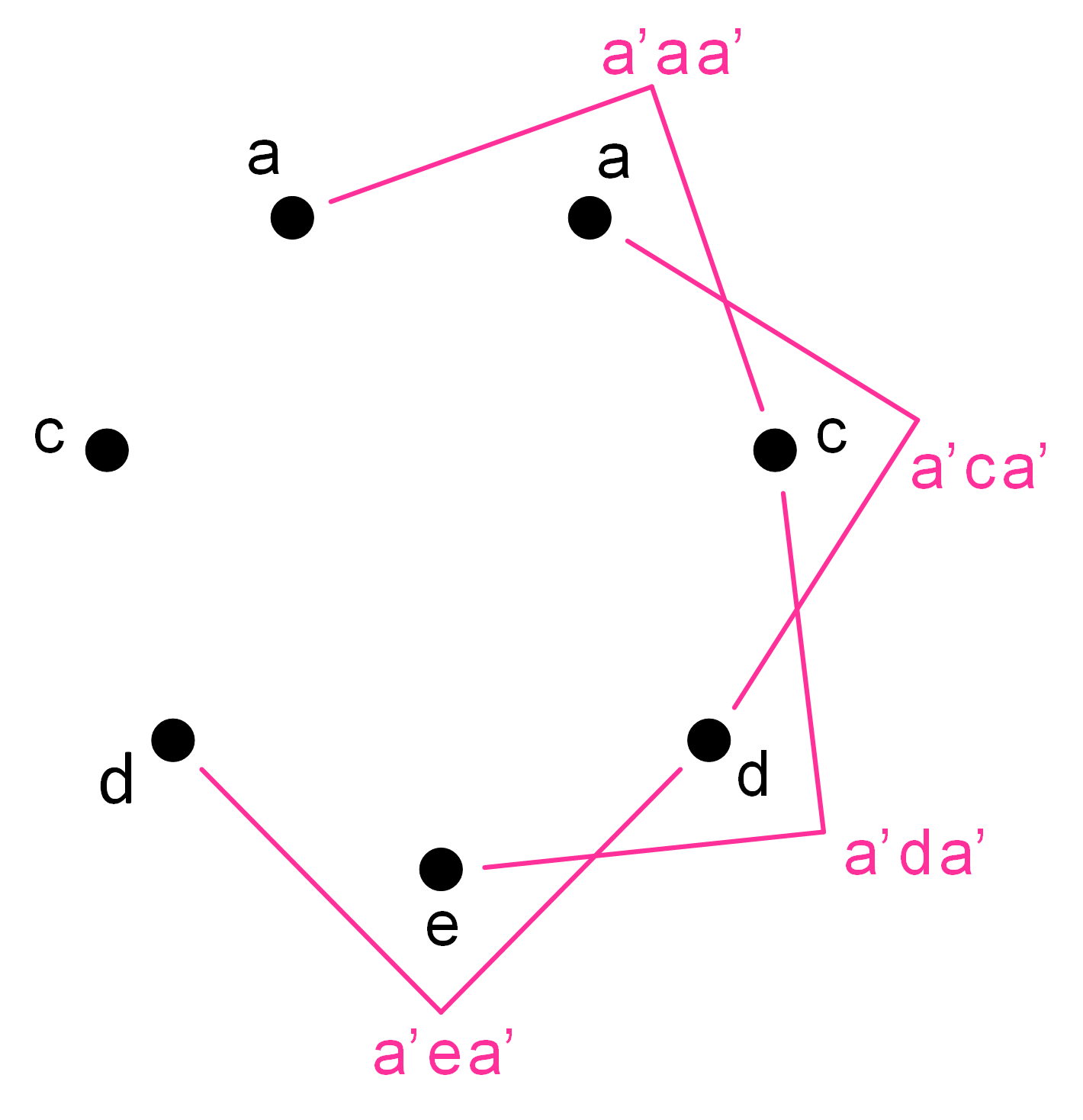}
         \caption{}
         \label{fig:S4-moves}
     \end{subfigure}
     
        \caption{Visualization of moves from $S_3$ to (a)  $S_1\cup S_3$ (b) $S_4$.}
        \label{fig:moves}
\end{figure}
 
To show that we cannot move from $S_3$ to $S_4$, we consider the possible  locations of the minmax peak of $\p'$. Due to symmetry of positions in $S_3$, the four peak configurations, shown in Figure~\ref{fig:S4-moves},  are: $a'aa'$ with sums $d+e \leq d+c$,  $a'ca'$ with sums $e+d=c+a$, $a'da'$ with sums $d+c>a+a$, or $a'ea'$ with sums $c+a$ (in both cases). Note that in the first three cases, we have $a'<a$ because the minimum of the minmax peak in $S_4$ has to be strictly less than the adjacent stacks, and in each of these cases, the $a$ stack is one of them. We can play on one more stack adjacent to the $a'$ stacks and we play on the stack that affects the larger sum.  In the first two cases, the peak sum is smaller than the smaller of the two sums, and since we can adjust only one sum, we cannot legally move to $\p' \in S_4$. For the third case, equality with the peak sum requires that  $d'+c = d+a'$ and hence $d' = d-c+a'<a'$ because $c>d$ by Remark~\ref{rem:props}(\ref{rem:props1}). For the last case,  the overlap stack is at the same location in $\p$ and $\p'$, so by Remark~\ref{rem:props}(\ref{rem:comsum}), we cannot adjust all four sums with play on only four stacks.   This shows that we cannot move to $\p' \in S_4$.

 {\bf Moves from $S_4$}: Last but not least, we check whether we can move from $\p=(a,b,c,d,e,a,g) \in S_4$ to $\p' \in S$.  The approach is similar to that when $\p\in S_3$. For a  move  to $\p' \in S_1 \cup S_3$, we once more need to create a tub configuration $xa'a'x$, where $a'\leq a$, and $a'=0$ for moves to $S_1$.  Due to the semi-symmetric nature of positions in $S_4$, we now need to consider all seven placements of the new pair of minima. We start by putting them at stacks $a$ and $b$ and get the following cases:  $x=c, x_1x_2x_3=aed$ (since we have to reduce $g$), $x=a, x_1x_2x_3=eag$, $x=\min\{b,e\}, x_1x_2x_3=aga$ (no matter which side we need to play on),  $x=a, x_1x_2x_3=bag$ (since we need to play on $c$), $x=d, x_1x_2x_3=abc$, $x=a, x_1x_2x_3=dcb$, and $x=a, x_1x_2x_3=cde$. 
 
 First we look at the cases where $x=a$. Reducing $x_1$ to zero, we have that  $x_2+x_3= a+g=c+b=e+d>a+a\ge a+a'$, so the sum conditions of $S_1$ and $S_3$ are not satisfied. Likewise,  for $x=c$, we have that $x_2+x_3=e+d=c+b>c+a \ge c+a'$, and for $x=d$, we obtain $x_2+x_3=b+c=d+e>d+a\ge d+a'$.  Finally, for $x=\min\{b,e\}$, we have that  $x_2+x_3=g+a=\min\{b,e\}+\max\{d,c\}>\min\{b,e\}+a\ge\min\{b,e\} +a'$, so we cannot move to $\p' \in S_1\cup S_3$.

 Next we look at moves from $S_4$ to $S_2$.  Since $a< \min\{b,e,g\}$, we have to reduce at least those three stacks to $a$ which requires play on   five stacks. Therefore we cannot move from $S_4$ to $S_2$.

 Finally, we look at moves from $S_4$ to $S_4$. If we keep the location of the minima and hence the overlap stack, then by Remark~\ref{rem:props}(\ref{rem:comsum}) there is no move to $\p' \in S_4$. Thus we need to consider whether we can create a minmax peak $a'Xa'$ with $a'<a$ and remaining stacks $\xx14$ which satisfy $x_1+x_2=x_3+x_4=a'+X$ by Remark~\ref{rem:props}(\ref{rem:peak}). We can play on either $x_1$ or $x_4$, but in either case we can only modify one of the two sums $x_1+x_2$ and $x_3+x_4$. The common sum for $\p$ is $s=g+a$, while the  for $\p'$ it is $s'=X+a'<s$. Furthermore, $x_2$ and $x_3$ cannot be adjusted. Let's look at the possible cases, going clockwise and starting with new minimia at the $g$ and $b$ stacks, for a total of six cases: (1) $X \le a$ and $\xx14=cdea$; (2) $X \le b$ and $\xx14=deag$; (3) $X \le c$ and $\xx14=eaga$; (4) $X \le d$ and $\xx14=agab$; (5) $X \le e$ and $\xx14=gabc$; and (6) $X \le a$ and $\xx14=abcd$.
 In cases (1) and (3), $x_3>X$, while in cases (4) and (6), $x_2>X$, either directly from the definition of positions in $S_4$ or by Remark~\ref{rem:props}(\ref{rem:S4vals}). For the remaining two cases, (2) and (5), we have that $x_1+x_2=x_3+x_4=g+a=s>s'$ and we can adjust only one of the two sums. This shows that there is no move from $S_4$ to $S_4$. 
 
 This completes the proof that there is no move from $S$ to $S$.
 \end{proof}

\subsection{There always is a move from \texorpdfstring{$\p \in S^c$ to $\p' \in S$}{\bf{p} not in S to \bf{p}' in S }}\label{sub:NtoP}

We now show the second part of Theorem~\ref{thm:howtoprove}.

\begin{proposition} 
If $\p \in S^c$, then there is a  move  to  $\p' \in S$.
\label{prop:wintolose}
\end{proposition} 

To show that we can make a legal move from any position $\p \in S^c$ to a position $\p' \in S$, we partition the set $S^c$ according to the number of zeros of $\p$ and, for positions without a zero stack, according to the number of maximal stacks and their location. Note that if $\p$ contains an empty stack, then we cannot move to $S_3$. Also, except for a move to the terminal position, we never are forced to move to $S_2$, even though the easiest move from a position that contains three consecutive minima is to $S_2$ (by making the other four stacks equal to that minimum height). We will only need to distinguish between the case of exactly one zero and the case of at least two zeros. Note that in~\cite{Vo2018}, $S^c$ was partitioned according to the exact number of minima of $\p$. The proof presented here is shorter and uses some of the ideas from~\cite{Vo2018}, such as  Definition~\ref{def:valley} and Lemma~\ref{lem:valley}. We call out these structures and $\CN{3}{2}$-equivalence (defined below) because they give insight into stack configurations from which it is easy to move to $\P$-positions. 

\begin{definition}
\label{def:valley}
A position $\p$ is called  \emph{deep-valley} if and only if five consecutive stacks $p_1p_2p_3p_4p_5$ satisfy $p_2+p_3+p_4 \leq \min\{p_1,p_5\}$. It is called \emph{shallow-valley} if and only if $p_1 \leq p_5$ and $p_2+p_3 \leq p_1 < p_2+p_3+p_4$.
\end{definition}
\vspace{0.1in}

\begin{lemma} [Valley Lemma] \label{lem:valley}
If $\p = (p_1,p_2,p_3,p_4,p_5,p_6,p_7)$ is \emph{deep-valley} and $s = p_2+p_3+p_4$, then there is a move to $\p' = (s,p_2,p_3,p_4,s,0,0)\in S_1$. On the other hand, if $\p$ is \emph{shallow-valley}, then there is a move to $\p'  = (p_1,p_2,p_3,p_1-(p_2+p_3),p_1,0,0)\in S_1$.
\end{lemma}

\begin{proof} 
If $\p$ is deep-valley, then $p_1'=p_5' =p_2+p_3+p_4 \leq \min\{p_1,p_5\}$, so it follows that $\p \rightarrow \p' \in S_1$  is a legal move. If $\p$ is shallow-valley, then  $p_1'=p_5'=p_1 \leq p_5$,  $p_1-(p_2+p_3) \geq 0$, and $p_4 \ge p_4'=p_1-(p_2+p_3)$. Also, $p_1-(p_2+p_3)+p_2+p_3 = p_1$, $\p \rightarrow \p' \in S_1$  is a legal move.
\end{proof}

The notion of $\CN{3}{2}$-equivalence comes into  play when $\p$ contains  zero stacks. It builds on the structure of the $\P$-positions of $\CN{3}{2}$, which are those with equal stack heights (see either \cite{mDsH09} or convince yourself easily with a one-line proof). Note that the definition below is not specific to the game $\CN{7}{4}$.  \\

\begin{definition}\label{def:cn32}
A position $\p$ of a $\CN{n}{k}$ game is \emph{$\CN{3}{2}$-equivalent} if the  stacks of $\p$ can be partitioned into subsets  $A_1$, $A_2$, and $A_3$ together with a set (or sets) of consecutive zero stacks, where $A_1$, $A_2$, and $A_3$ satisfy the following conditions:
\begin{enumerate}
    \item $A_i \cap A_j=\varnothing$ for $i \ne j$;
    \item Any pair of the three sets $A_1$, $A_2$, and $A_3$ and any zero stacks that are between them are contained in $k$ consecutive stacks;
    \item Any move that involves at least one stack from each of the three sets $A_1$, $A_2$, and $A_3$ requires play on at least $k+1$ consecutive stacks, thus is not allowed.
\end{enumerate}

We define the {\em set sums } $\tilde{p}_i=\sum_{p_j\in A_i}p_j $ and call a move a  \emph{$\CN{3}{2}$ winning move} if play on the stacks in the sets $A_i$ results in equal set sums in $\p'$. A $\CN{3}{2}$-equivalent position that has equal set sums is called a $\CN{3}{2}$-equivalent \P-position.
\end{definition}
\vspace{0.1in}

$\CN{3}{2}$-equivalent positions are custom-made for moves to $S_1$ since the conditions on the non-zero stacks require equality of the tri-sum and the two adjacent stack heights (set sum of a single stack). But we will also see that a   $\CN{3}{2}$ winning move can be used when there are additional inequality conditions on some of the stacks as long as those conditions can be maintained. In other instances, the  sum conditions may involve a stack outside the three sets, but the sum condition can be achieved without play on that ``outside'' stack.

\vspace{0.1in}
The proof of Proposition~\ref{prop:wintolose} will proceed as a sequence of lemmas where we will consider the individual cases according to the number of zeros and number and location(s) of the  maximum values in the case when the position does not have a zero. We start by dealing with positions that have at least two zero stacks.

\begin{lemma} [Multiple Zeros Lemma] If $\p \in S^c$ and $\p$ has at least two stacks without tokens, then there is a move to $\p' \in S_1 \cup S_2 \cup S_4$. 
\label{lem:zeros}
\end{lemma}

\begin{proof} Note that we will label the individual stacks as $x$, $x_i$, $y$, and $y_j$ depending on the symmetry of the position as well as the role the different stacks play. Typically, stacks labeled $x$ or $x_i$ are between zeros (short distance) or adjacent to zeros. Since the positions in $S_1\cup S_3 \cup S_4$ all have sum conditions that need to be satisfied, we will typically use $s$ to denote this target sum. We consider the case of two adjacent zeros, two zeros separated by one stack and finally two zeros separated by two (or three) stacks. Figure~\ref{fig:multizero} shows the generic positions in each of the cases.

\begin{figure}[ht]
     \centering
     \begin{subfigure}[b]{0.3\textwidth}
         \centering
         \includegraphics[width=0.8\textwidth]{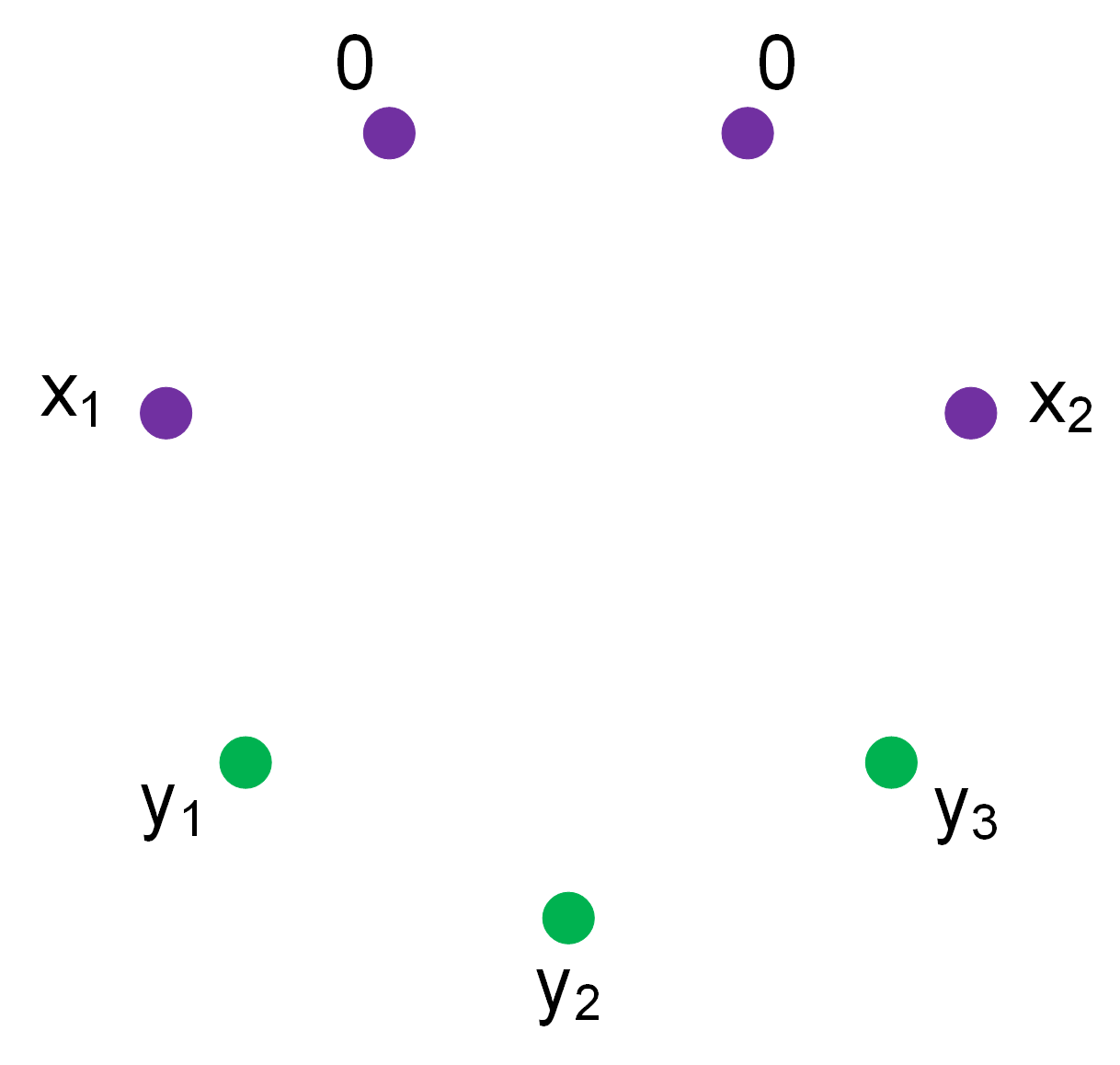}
         \caption{}
         \label{fig:2cons}
     \end{subfigure}
     \hfill
     \begin{subfigure}[b]{0.3\textwidth}
         \centering
         \includegraphics[width=0.8\textwidth]{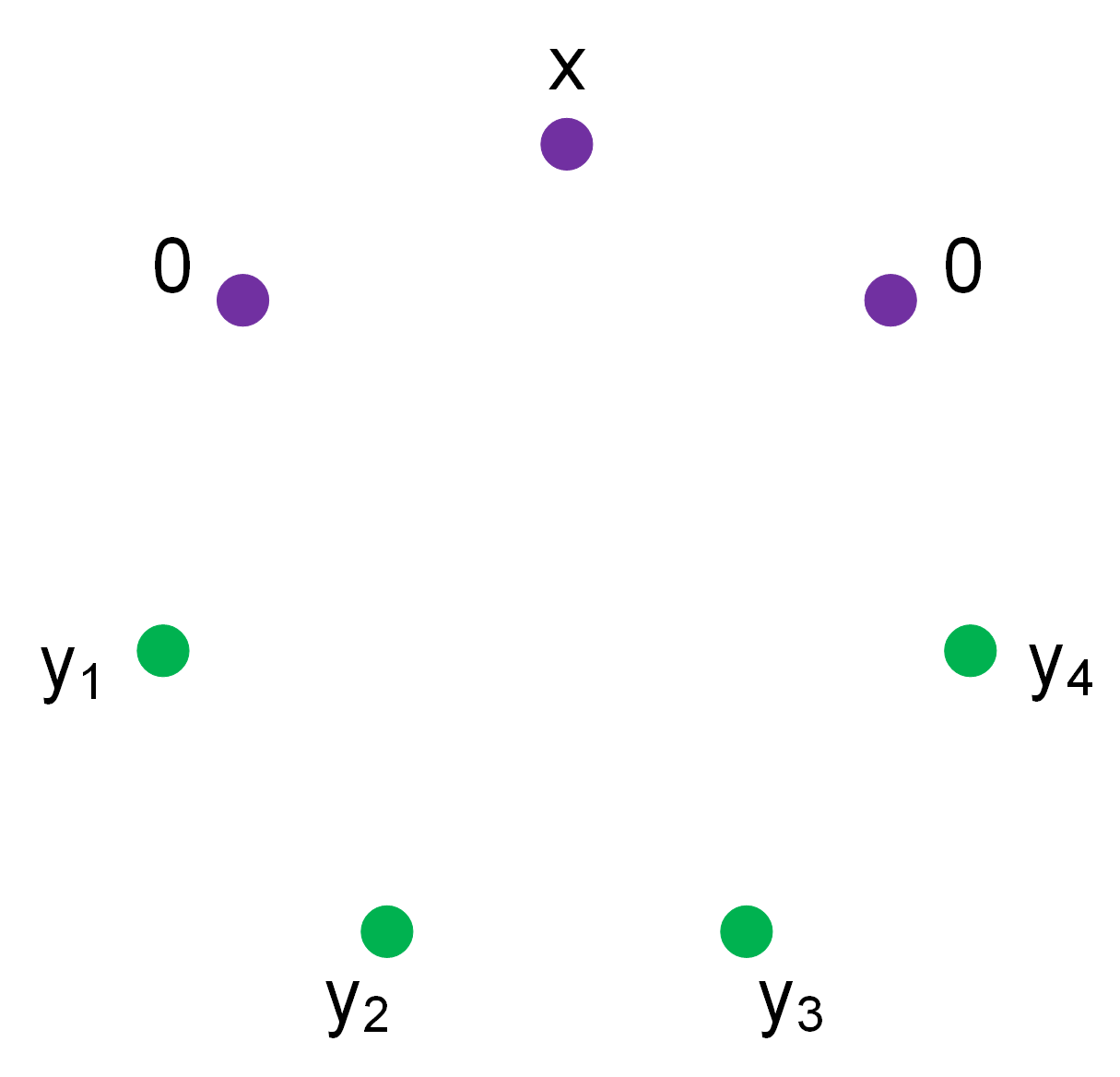}
         \caption{}
         \label{fig:2oneApart}
     \end{subfigure}
       \hfill
     \begin{subfigure}[b]{0.3\textwidth}
         \centering
         \includegraphics[width=0.8\textwidth]{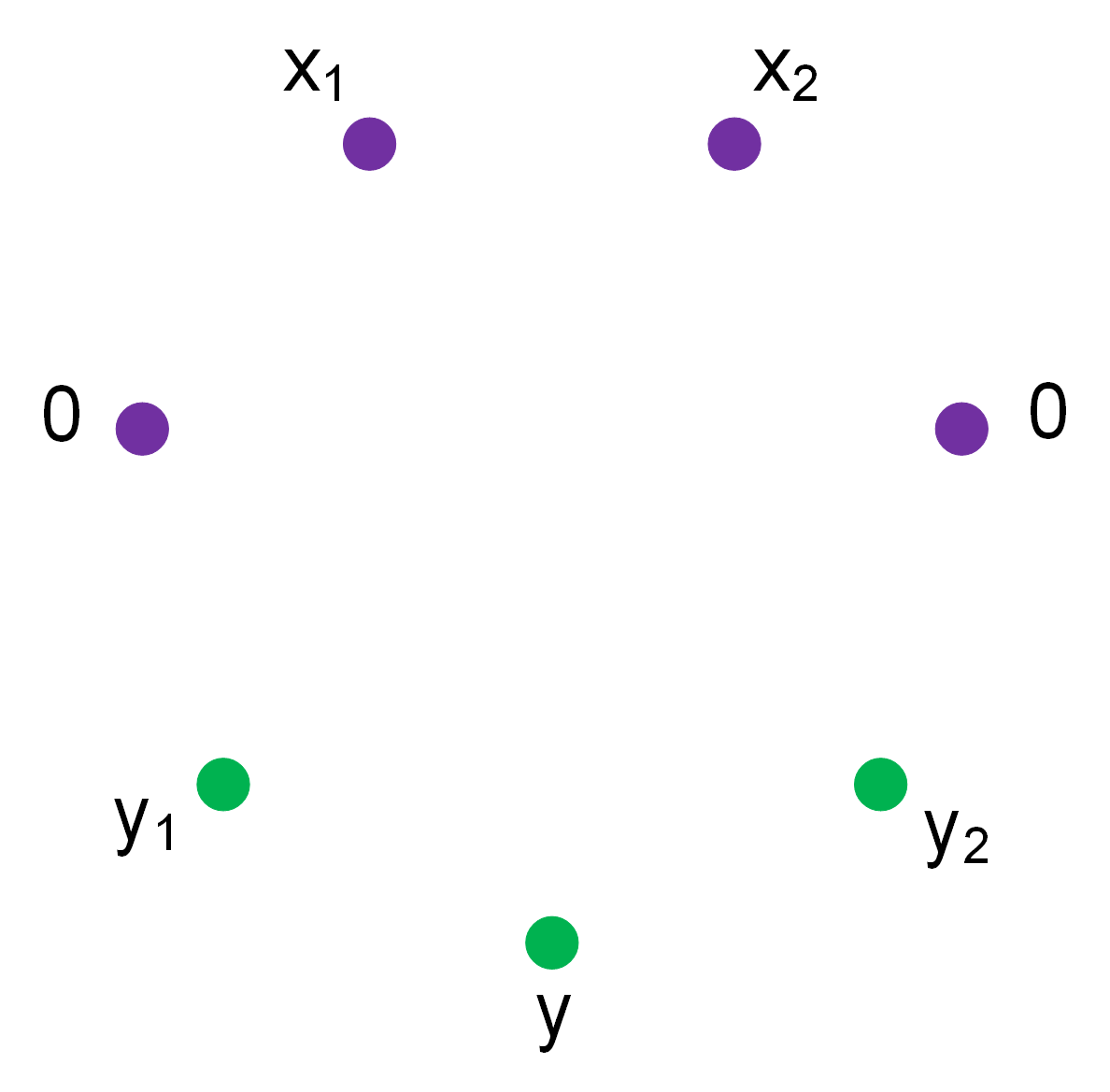}
         \caption{}
         \label{fig:2twoApart}
     \end{subfigure}
     
        \caption{Generic positions with at least two zeros. (a) Two consecutive zeros. (b) Two zeros separated  by one stack. (c) Two zeros separated by two stacks.}
        \label{fig:multizero}
\end{figure}

First, suppose there are two consecutive zeros in the position, then $\p = (x_1,0,0,x_2,y_3, y_2,y_1)$, shown in Figure~\ref{fig:2cons}. Note that  $\p$ is $\CN{3}{2}$-equivalent with sets $A_1=\{x_1\}$, $A_2=\{x_2\}$, and $A_3=\{y_1, y_2, y_3\}$. Thus we can make the $\CN{3}{2}$ winning move to $\p'\in S_1$ by adjusting the stacks in two of the $A_i$ to make the set sums in $\p'$ equal to the minimal set sum in $\p$. This can be achieved with play on four stacks or fewer.

Now we can assume that any zeros in $\p$ are isolated, that is, they are either separated by one stack  or by two stacks in their shortest distance between them. Let's first consider the case of two zeros separated by a single stack, that is, $\p=(0,x,0,y_1,y_2, y_3, y_4)$ with $\min\{x, y_1,y_4\}>0$ because of the isolated zero condition (see Figure~\ref{fig:2oneApart}). Our goal is to move to $S_4$. Due to the zeros, the sum conditions of $S_4$ reduce to $x'=y_1'+y_2' = y_3'+y_4'$, with $\min\{y_1',y_4'\}>0$, so $\p$ is $\CN{3}{2}$-equivalent with sets $A_1=\{x\}$, $A_2=\{y_1,y_2\}$, and $A_3=\{y_3, y_4\}$ and we can make the $\CN{3}{2}$ winning move to $\p'$. Note that we can achieve the condition $\min\{y_1',y_4'\}>0$ because the original stacks were non-zero, and any set of two stacks that is being played on can be adjusted to achieve the desired sum without making $y_1$ or $y_4$ equal to zero since $x>0$ by assumption of the isolated zeros. However, if in the process, we need to make $y_2'=y_3'=0$, then the resulting position is in $S_1$.

Now we turn to the case where the zeros are separated by two stacks, that is, $\p=(0,x_1,x_2,0,y_2,y, y_1)$, with $\min\{x_1,x_2, y_1, y_2\} >0$ since we assume isolated zeros (see Figure~\ref{fig:2twoApart}).  We also assume w.l.o.g. that $y_2 \ge y_1$. Now we need to consider two subcases: $y_1 \ge x_1$ and $y_1 < x_1$. Note that for each of the subcases, the sum $s$ will be defined on a case by case basis. 

In the first case, we let $s= \min\{x_1+x_2,y_1\}$ and move to $\p'=(0,x_1,x_2',0,s,0, s)\in S_4$ with $x_1+x_2'=s$.  While this looks like there is play on five stacks, either $x_2$ or $y_1$ will remain the same. If $s = y_1$, then play is on the $x_2,0,y_2$ and $y$ stacks, and because $x_1 \le y_1$, we have $x_2'=s-x_1=y_1-x_1\ge 0$. If $s = x_1+x_2$, then play is on the three $y$ stacks.

Now we look at $y_1 < x_1$, which is a little bit more involved. Here our goal is to move to $S_1$, so we need to create a pair of zeros. Since $y_2 \ge y_1$, we choose $x_2'=0$ and show that we can make $x_1'$, $y_2'$, and the tri-sum $0+y_1'+y'$ equal in $\p'$. Let $s = \min\{x_1,  y_1+y,y_2\}$. If $s = x_1$, $ s=y_1+y $, or $s = y_2$ with the additional condition that  $y\le y_2$, then we can move to $\p'= (s,0,0,s,y',y_1',0)$ with $y'+y_1'=s$ by playing on at most four stacks. If $s = x_1$, then play is on stacks $x_2$, $y_2$, and $y$, with $y'=s-y_1=x_1-y_1>0$. If $s = y_1+y$, then play is on stacks $x_1$, $x_2$, and  $y_2$. Finally, if $s=y_2 \ge y$, then play is on stacks $x_2$, $x_1$, and $y_1$, with $y_1'=y_2-y \ge 0$. 

This leaves the case of $y_1 <x_1$, $y_1 \le y_2$,   $y_2 <\{x_1, y_1+y\}$ with $y > y_2$ unresolved. This set of inequalities can be simplified to $y_1 \le y_2$, $y_1 <x_1$,  and $y_2 <\{x_1, y\}$. Note specifically that $y > y_i$ for $i = 1,2$. We need to make further distinctions as to where the maximal value occurs. In all cases we will move to $S_1$, but the location of the maximal value determines where the pair of adjacent zeros is created. Let $M=\max(\p)=\max\{x_1,x_2, y\}$ (all other stacks cannot be maximal due to the inequalities).

First we consider the case where the maximal value occurs next to a zero, that is, $M=x_1$ or $M= x_2$. Let  $s=\min\{x_1+y_1, x_2+y_2, y\}$ and assume that $M=x_1$. We claim that there is a legal move to $\p' \in S_1$ where $\p'=(0,s,x_2',0,y_2, s, 0)$ with $x_2'+y_2=s$. Note that $M=x_1$ implies that $s <x_1+y_1$ because $s =x_1+y_1$ leads to a contradiction; since $y_i>0$ due to isolated zeros, we would have $x_1<x_1+y_1=s\le y \le M=x_1$. If $s = x_2+y_2$, then play is on stacks $x_1$, $0$, $y_1$, and $y$ and it is a legal move since $x_1=M \ge y\ge s$. If $s = y$, then play is on stacks $y_1$, $0$,  $x_1$, and $x_2$, with $x_2'=s-y_2=y-y_2>0$. Since $y > y_i$, the same proof, except with subscripts 1 and 2 changing places, applies when  $M=x_2$.

The final case is when $M=y >\max\{x_1,x_2\}$. We first consider $x_1>x_2$ and let $s = \min\{x_1, x_2+y_2\}$.  Then the move is to $\p'=(0,s,x_2,0,y_2', s, 0) \in S_1$ with $x_2+y_2'=s$. If $s=x_1$, then play is on $y_1$, $y$, and $y_2$. The move is legal since $y > x_1$ and $y_2'=x_1-x_2>0$. On the other hand, if $s =  x_2+y_2$, then play is on stacks $y$,  $y_1$, and $x_1$ and $y>x_1>s$.
This completes the case of two zeros that are two stacks apart, and therefore, the case of more than two zeros. 
\end{proof}

We next consider the case of a single isolated zero.

\begin{lemma} [Unique Zero Lemma]
If a position $\p\in S^c$ has a unique zero, then there is a move to $\p' \in S$.
 \end{lemma}
 
 \begin{figure}[htb]
    \centering
    \includegraphics[scale=0.40]{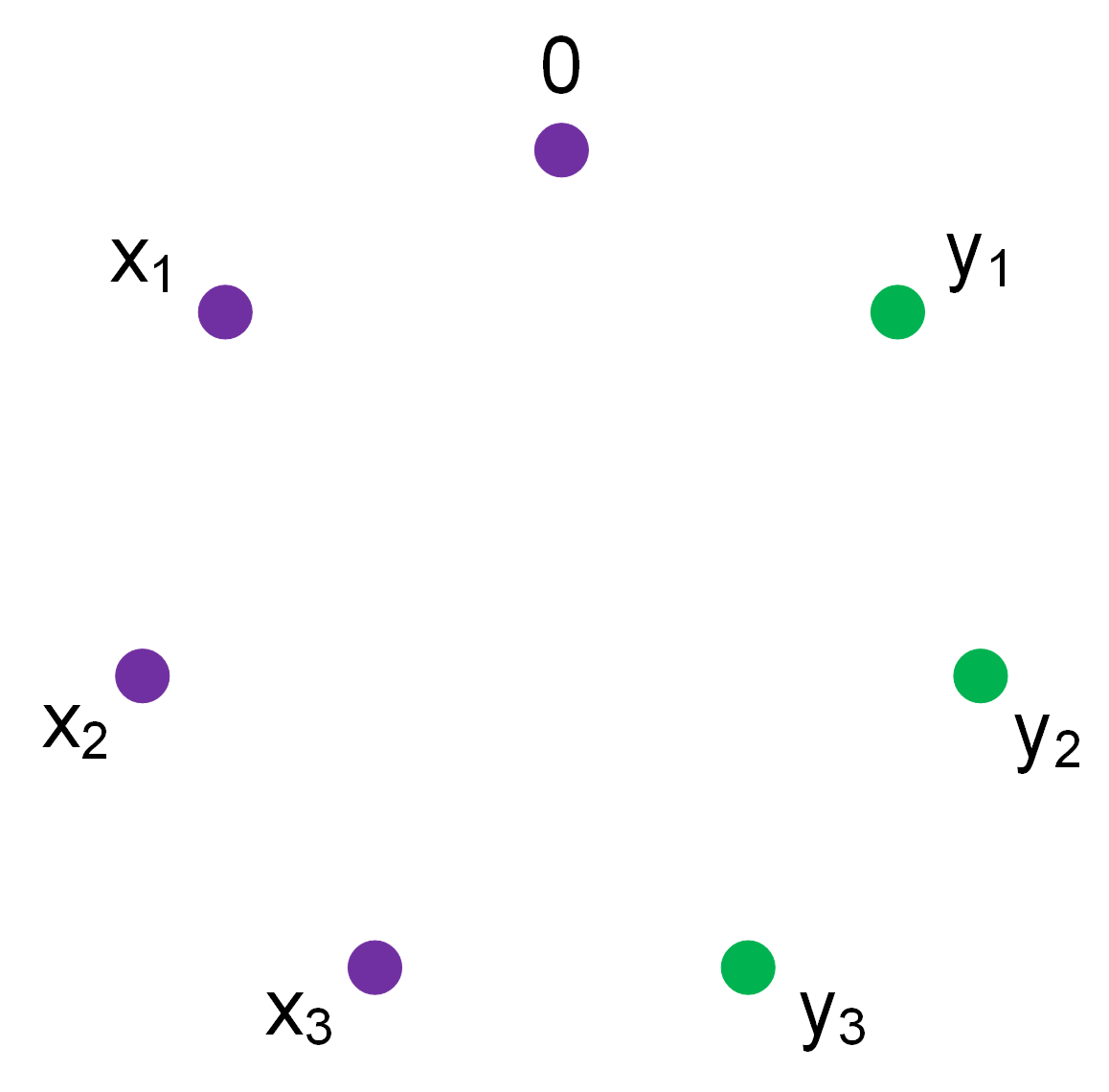}
    \caption{Generic position with a unique zero.}
    \label{fig:isozero}
\end{figure}

 \begin{proof}
 The generic position for this case is shown in Figure~\ref{fig:isozero}. Note that due to the assumption of the unique zero, we have that all other stack heights are non-zero, so $x_i>0$ and $y_i>0$ for $i = 1, 2, 3.$ We may also assume w.l.o.g. that $x_2 \ge y_2$. We will see that in almost all cases, we can move to $S_1$; there is a single subcase where we will move to $S_4$. Table~\ref{tab:isozero} gives a quick overview of the structure of the subcases. 
 
 \renewcommand{\arraystretch}{1.5}
 \begin{table}[h!]
\centering
  \begin{tabular}{|l|l|l|c|} 
\hline
   $x_1+y_1 \le \min\{x_2, y_2\}=y_2$ & \multicolumn{2}{l|}{}& (a) \\ \hline
   \multirow{3}{*}{$x_1+y_1 >y_2$} & $y_2 \ge y_1$& & (b) \\ \cline{2-4}
   &\multirow{2}{*}{$y_2 < y_1$}& $x_2 \ge y_1$ &  (c)
   \\\cline{3-4}
   & & $x_2 < y_1$  & (d)
   \\ \hline
 \end{tabular}
 \vspace{0.1in}
\caption{Subcases for unique zero.}
  \label{tab:isozero}
 \end{table}
 
 \begin{itemize}
 \item[(a)] If $s = x_1+y_1 \le \min\{x_2, y_2\}$, then we can move to $\p'=(0,x_1, s, 0, 0, s, y_1) \in S_1$. 
 \item[(b)] When  $x_1+y_1 > y_2 \ge y_1$,  we have that $y_2y_10x_1x_2$ is a shallow valley and by the Valley Lemma, there is a move to $S_1$. 
 \item[(c)] Since $y_1 > y_2$ implies that $x_1+y_1 > y_2$, the conditions reduce to $y_1 > y_2$, $x_2 \ge y_1$ and $x_2 \ge y_2$. Let $s = \min\{y_1, y_2+y_3+x_3\}$. The goal is to keep stacks $y_2$ and $0$ and then adjust the other stacks according to the value of $s$. 
If $s = y_1$, then we move to $\p'=(0,y_1,y_2,y_3',x_3',s,0)\in S_1$ with $y_2+y_3'+x_3'=s=y_1$,  otherwise, we move to $\p'=(0,y_1',y_2,y_3,x_3,s,0)\in S_1$ with $y_1'=s=y_2+y_3+x_3$. These moves are legal because  $x_2\ge y_1\ge s$ and $y_3'+x_3'=s-y_2=y_1-y_2>0$. 
 \item[(d)] 
The conditions for this case, namely $y_2 < y_1, x_2 < y_1$, and $x_2 \ge y_2$ reduce to $y_2 \le x_2 <y_1$. We distinguish between two main cases, namely whether $x_3+y_3 \le \min\{x_1, y_1\}$ or not. We first consider the case $x_3+y_3 \le \min\{x_1, y_1\}$.
 \begin{itemize}
     \item If  $y_2 < s=x_3+y_3 \le \min\{x_1, y_1\}$, then we can move to $\p'=(0, s-y_2, y_2, y_3, x_3, 0 , s) \in S_4$. Since $\min\{s-y_2,y_2, x_3\}>0$, the conditions of $S_4$ are satisfied. 
     \item If  $s=x_3+y_3 \le y_2\le x_2$, then $x_2x_3y_3y_2y_1$ is either a shallow valley or a deep valley, depending on whether $x_3+y_3+y_2> x_2$ or $x_3+y_3+y_2 \le x_2$, and there is a move to $S_1$.
 \end{itemize}
 Now we look at the second case, $x_3+y_3 > x_1$ or  $x_3+y_3 >y_1$. We show that with this condition alone (disregarding the overall conditions of subcase d), we can show that there is a move to $S_4 \cup S_1$. We can therefore assume,  w.l.o.g, that $x_1 \ge y_1$, and 
consider two subcases, namely $x_1 \ge x_3+y_3 > y_1$ and $x_3+y_3 >x_1$. 
\begin{itemize}
    \item If $x_1 \ge x_3+y_3 > y_1$ and $x_3+y_3 > y_1+y_2$, then we can move to $\p'=(0,y_1,y_2, y_3',x_3',0,s)\in S_4$ with $s=y_1+y_2=y_3'+x_3$. We can adjust the sum $y_3'+x_3'$ such that $x_3'>0$. Also, $\min\{y_1, y_2\}>0$, so the $S_4$ conditions are satisfied. If, on the other hand,  $x_3+y_3 \le  y_1+y_2$, then we can move to $\p'=(0,y_1',y_2, y_3,x_3,0,s)\in S_4$ with $s=x_3+y_3$ and $y_1'=s-y_2>0$ and the $S_4$ conditions are satisfied.
    \item If $ x_3+y_3 > \max\{x_1, y_1\}$ and $x_1\ge y_1+y_2=s$, then we can move to $\p'=(0,y_1,y_2, y_3',x_3',0,s)\in S_4$ with $s=y_1+y_2=y_3'+x_3$. Note that once more, $\min\{x_1, x_3+y_3\}\ge y_1+y_2=s$, so the move is legal. Finally, assume that $y_1+y_2 > x_1=s $. Now we have a move to $\p'=(0,y_1,y_2',y_3',x_3',0,x_1)\in S_4\cup S_1$ with $y_1+y_2'=x_3'+y_3'=s$. Since $y_1 \le s$, we can make the sum $y_1+y_2'=s$, and we can also adjust the sum $x_3'+y_3'$ while keeping $x_3'>0$. If $y_3'=y_2'=0$, then $\p'\in S_1$, otherwise $\p'\in S_4$.
\end{itemize}

\end{itemize} 
\noindent This completes the proof in the case of exactly one zero.
\end{proof}

Finally, we deal with the case when the position $\p$ does not have a zero. In this case, we divide the positions according to where the maximum is located in relation to other maxima (if any). Note that 
when $\min(p)>0$, there is a close relation between positions in $S_3$ and $S_4$. A position $\p=(m,M,m,p_4,p_5,p_6,p_7)$ with $p_4+p_5=p_6+p_7=M+m$ and $\min\{p_4,p_7\}>m$ is in $S_4$ if $\max\{p_5,p_6\}>m$ and is in $S_3$ if $p_5=p_6=m$. Therefore, we will state that there is a move to $S_3\cup S_4$ and need only check on the sum conditions and the minimum condition. This property will be used repeatedly in the Maximum Lemma.\\

\begin{lemma}[Maximum Lemma] Let $\p \in S^c$ with $\min(\p)>0$. Then there is a move from $\p$ to $\p'\in S$.
 \end{lemma}
 
 \begin{proof}
 Let $M=\max(\p)$. We will first look at the antipodal case, where we have two maxima opposite of each other. The generic position is $\p=(x_1,x_2, M, y_3,y_2,y_1, M)$, shown in Figure~\ref{fig:apmax2M}.  
 
 \begin{figure}[ht]
     \centering
     \begin{subfigure}[b]{0.4\textwidth}
         \centering
         \includegraphics[width=0.75\textwidth]{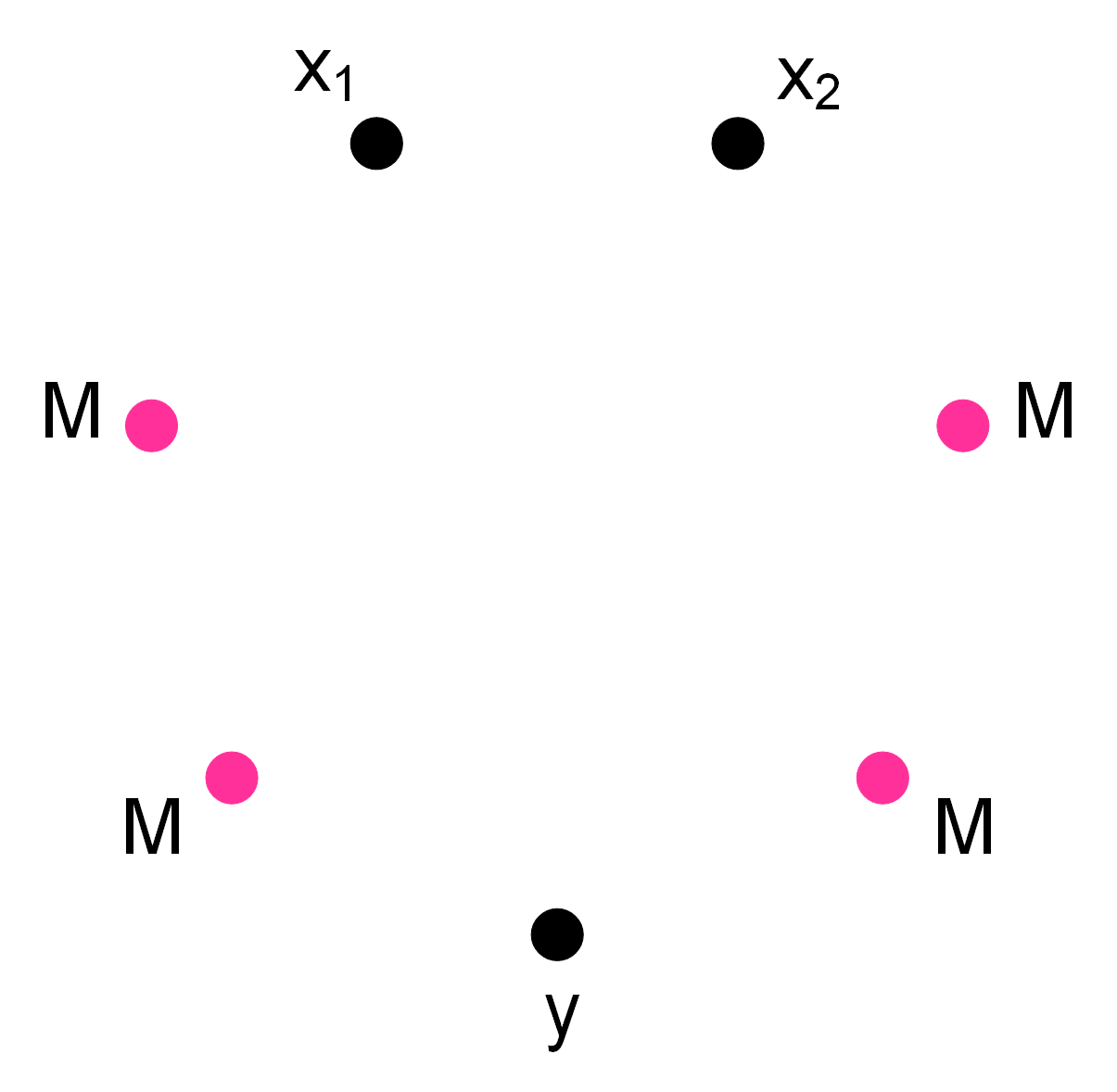}
         \caption{}
         \label{fig:apmax4M}
     \end{subfigure}
     \hspace{0.05\textwidth}
     \begin{subfigure}[b]{0.4\textwidth}
         \centering
         \includegraphics[width=0.75\textwidth]{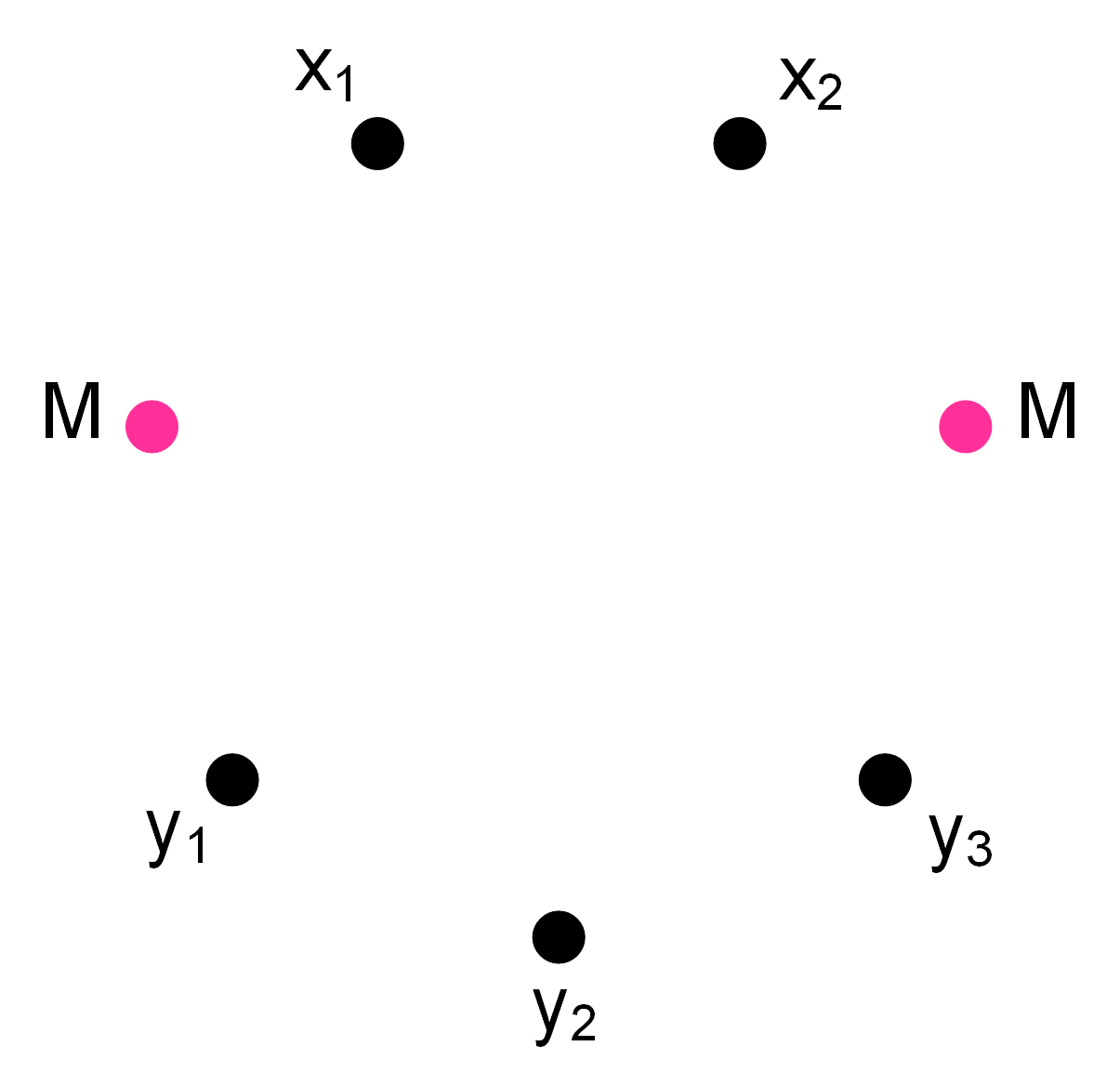}
         \caption{}
         \label{fig:apmax2M}
     \end{subfigure}
     
        \caption{Generic positions for antipodal maxima. (a) $y_3 = M$ and (b) $y_3 < M$.}
        \label{fig:apmax}
\end{figure}

 Table~\ref{tab:apmax} shows the subcases we will consider for  antipodal maxima. Without loss of generality, we may assume that $y_3 \le y_1$.

  \renewcommand{\arraystretch}{1.5}
 \begin{table}[h!]
\centering
  \begin{tabular}{|l|l|l|c|} 
\hline
   $y_3=M$ & \multicolumn{2}{l|}{}& (a) \\ \hline
   \multirow{3}{*}{$y_3<M$} & $y_2+y_3\le M$& & (b1) \\ \cline{2-4}
   &\multirow{2}{*}{$y_2+y_3 > M$}& $x_1 \ge x_2$ &  (b2)
   \\\cline{3-4}
   & & $x_1 < x_2$  & (b3)
   \\ \hline
 \end{tabular}
 \vspace{0.1in}
\caption{Subcases for antipodal maxima.}
  \label{tab:apmax}
 \end{table}

\begin{itemize}
    \item[(a)] We start with the case $M=y_3=M \le y_1$ shown in Figure~\ref{fig:apmax4M}. In this case, the generic position becomes $\p=(x_1,x_2,M,M,y,M,M)$, where we have dropped the $y$ subscript for ease of notation. We may also assume in this case that w.l.o.g., $x_1 \le x_2$. If $x_1+x_2 < M$, then $Mx_1x_2MM$ forms a shallow valley and there is a move to $S_1$. Now assume that $M \le x_1+x_2 \le M+y$. In this case, there is a move to $\p'=(x_1,x_2,x_1,M,x_1+x_2-M,x_1+x_2-M,M ) \in S_3$. We can make the necessary adjustments since $x_1 \le M=\max(\p)$, and $M \ge y \ge x_1+x_2-M \ge 0 $   by assumption. Finally, when $M+y < x_1+x_2$, then we can move to $\p'=(x_1',x_2',y,M,y,M,y)\in S_4$, with $x_1'+x_2'=M+y$. Note that $M+y< x_1+x_2$ implies that $M>y$. We need to show that we can adjust the $x_1$ and $x_2$ stacks such that $x_1'>y$ and $x_2'>y$ to satisfy the $S_4$ conditions. This is possible since $x_1+x_2>M+y\ge y+1+y=2y+1$.
      \end{itemize}
      We now assume that $M>y_3$ (see Figure~\ref{fig:apmax2M}) and consider the various subcases listed in Table~\ref{tab:apmax}.
      \begin{itemize}
    \item[(b1)] Since $M \ge y_2+y_3$, position \p~is either shallow valley (if $y_1+y_2+y_3 >M$) or deep valley (if $y_1+y_2+y_3  \le M$), so there is a move to $\p' \in S_1$. 
    \end{itemize}
    \noindent Now let $s = \min\{y_2+y_3,M+x_1, M+x_2\}$.
    \begin{itemize}
    \item[(b2)]  If $s = y_2+y_3$ or $s = M+x_2$, then there is a move to $\p' = (s-M, s - M, M, y_3, y_2',y_3, M) \in S_3$ with $y_2' = s- y_3$. Note that in either case, we only play on four stacks. If $s = y_2+y_3$, then $s \le M+x_2 \le M+x_1$, so $s-M\le \min\{x_1,x_2\}$ and and $y_3 \le y_1$ by assumption. Also, $y_2'=s-y_3=y_2$, so play is  on the $x_2, x_1, M,$ and $y_3$ stacks. Since $M>y_3$, we have that $s-M=y_2-(M-y_3)<y_2$ as needed for positions in $S_3$.
    \item[(b3)] If $s = M+x_1$, then  $M+x_1 \le y_2+y_3$. We move to $\p'=(x_1,x_2, s-x_2,y_3',y_2',x_1,M)\in S_3\cup S_4$ with $y_2'+y_3'=M+x_1=s$, playing on the one of the $M$ stacks and the $y_i$ stacks. This move is legal because $y_1 \ge y_3 \ge M+x_1 - y_2 \ge x_1$ and $s-x_2 = M+x_1-x_2<M$. Left to show is that $\min\{x_2,y_2'\}>x_1$. By assumption of this case, $x_2 > x_1$, and $0<M-y_3\le y_2-x_1$ shows that we can satisfy the sum condition with $y_2'>x_1$.
\end{itemize} 
\noindent This completes the case of antipodal maxima. We now consider the case when $M > \max\{x_3, y_3\}$, so the stacks that are ``opposite'' of $M$ have strictly smaller height. Our generic position is shown in Figure~\ref{fig:nonapmax}. W.l.o.g., we may assume that $x_1 \le y_1$. Once more we move to either $\p'\in S_1$ or $\p' \in S_3\cup S_4$.  

\begin{figure}[htb]
    \centering
    \includegraphics[scale=0.45]{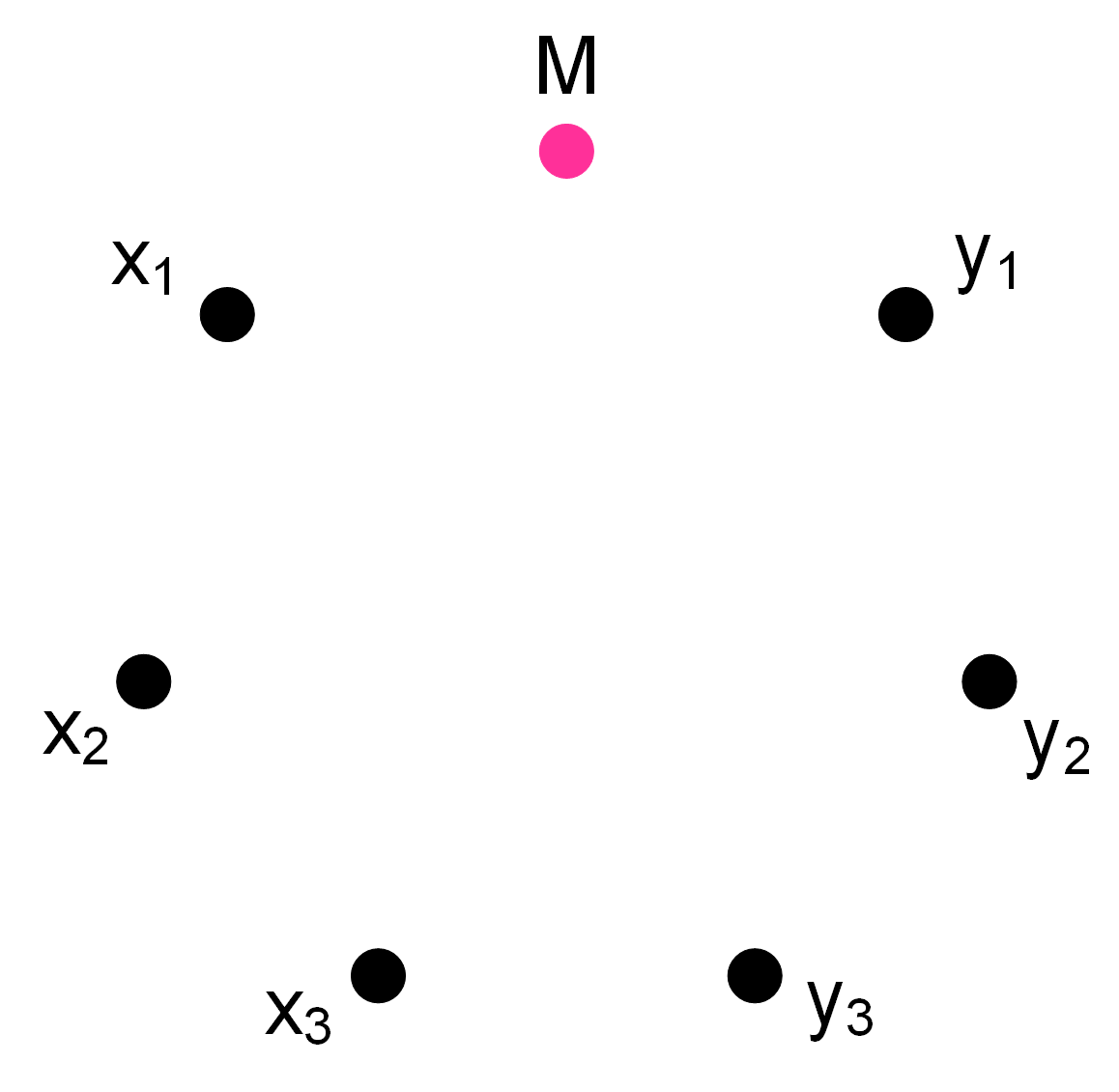}
    \caption{Generic position when $M > \max\{x_3, y_3\}$.}
    \label{fig:nonapmax}
\end{figure}

Let $s = \min\{M+x_1,x_2+x_3, y_2+y_3\}.$  

\begin{itemize}
    \item If $s= M+x_1$, then we can move to  $\p'=(M, x_1, y_2, y_3',x_3', x_2, x_1) \in S_3 \cup S_4$, with $x_2+x_3'=y_2+y_3'=M+x_1$. Play is on the $y_i$ stacks and $x_3$; the move is legal because $x_1 \le y_1$ by assumption, $x_3'=M+x_1-x_2 \le x_3$, and $x_3' >0$ since $M=\max(\p)$ and all stack heights are positive. Likewise, $0 < y_3' \le y_3$. Left to show is that $\min\{x_2,y_2\}>x_1$. By assumption, $M>\max\{x_3, y_3\}$ which implies both  $0<M-x_3\le x_2-x_1$ and $0<M-y_3\le y_2-x_1$, so the move is legal. 
  \item If $s= x_2+x_3$ and $y_3 \ge s = x_2+x_3$, then $M>y_3$ implies that $\p$ is either shallow valley (if $ y_3 < x_1+x_2+x_3$) or deep valley (if $ y_3 \ge x_1+x_2+x_3$). If  $y_3 < s = x_2+x_3<M+x_1$, then we move to $\p'=(M',m',y_2',y_3,x_3,x_2,m') \in S_4$ with overlap stack $M'$ and $y_2'+y_3=s$, where $M'=s, m' = 0$ if $M \ge s$ and $M'=M, m' = s-M$ otherwise. Let's check that this move is legal. If $M \ge s$, then we can clearly create the $M'$ and $m'$ stacks. If $M < s$, then $m'=s-M >0$ and $s-M<x_1\le y_1$, so that adjustment is legal. Next we consider the $y_2$ stack. Since $y_3 < s$ and $y_3<M$, then $y_2'=s-y_3>\min\{0,s-M\}$, so $y_2'>m'\ge0$. Last but not least, $x_2>0$ (by assumption of no zero stacks) and $x_2>x_2+x_3-M=s-M$ since $x_3<M$, so $x_2 > m'$. 
  \item If $s= y_2+y_3$, then the same arguments apply as in the case $s= x_2+x_3$, with the roles of $x$ and $y$ interchanged except for the inequality that $s <M+x_1$. 
  \end{itemize}
This completes the proof of the max lemma. 
 \end{proof}

With these three lemmas under our belt, we have proved Proposition~\ref{prop:wintolose}, because each position either has multiple zeros, a unique zero, or no zero. In each case, we have shown that there is a legal move from $\p \in S^c$ to $\p' \in S$. Together with Proposition~\ref{prop:noloselose} and Theorem~\ref{thm:howtoprove}, we have shown that the set $S$ of Theorem~\ref{thm:CN(7,4)} is the set of $\P$-positions of \CN{7}{4}.

\section{Discussion}

Our goal in the investigations of \CN{n}{k} has always been to find a general structure of the \P-positions for families of games. So far we have found such results for \CN{n}{1}, \CN{n}{n}, and \CN{n}{n-1} (see~\cite{mDsH09}). In addition, in all known results for \CN{n}{k}, we have been able to find a single description of the \P-positions. The case of \CN{7}{4} is seemingly an anomaly in that \sout{we had} four different sets \sout{that} make up the \P-positions. However, looking at the \P-positions of \CN{3}{2}, \CN{5}{3}, and \CN{7}{4}, which are all examples of \CN{2\ell+1}{\ell+1}, we found one commonality. Recall that the \P-positions of \CN{3}{2} are given by $\{a,a,a\}$ for $a \ge 0$, and the \P-positions of \CN{5}{3} are given by $\{(x,0,x,a,b)|x=a+b\}$. This leads to the following result.

\begin{lemma}\label{lem:genppos} In the game  \CN{2\ell+1}{\ell+1}, the set of 
 $\P$-positions contains the set $S_1$, where 
$$ S_1=\{\p=(x,\underbrace{0,\ldots,0}_{\ell-1},x,a_1,\ldots,a_{\ell})|\sum_{i=1}^{\ell} a_i=x\}. $$
\end{lemma}

\begin{proof}
Note that all positions in \CN{2\ell+1}{\ell+1} that have $\ell-1$ consecutive zeros are $\CN{3}{2}$-equivalent with sets $\{p_1\}$,$\{p_{\ell+1}\}$ and $\{p_{\ell+2},\ldots,p_{2\ell+1}\}$.  Those in $S_1$ are precisely the  $\CN{3}{2}$-equivalent \P-positions. Therefore, we cannot make a move from $S_1$ to $S_1$ because this would amount to a move from a \P-position in $\CN{3}{2}$ to another \P-position in $\CN{3}{2}$. On the other hand, we can make a \CN{3}{2} winning move into $S_1$ from any position in \CN{2\ell+1}{\ell+1} that has $\ell-1$ consecutive zeros. Therefore, $S_1$ must be a subset of the \P-positions of \CN{2\ell+1}{\ell+1}.
\end{proof}

While Lemma~\ref{lem:genppos} does not settle the question regarding the set of $\P$-positions of the family of games \CN{2\ell+1}{\ell+1}, the result shows that the set $S_1$ for \CN{7}{4}, which has the requirement of the zero minima, is not an anomaly, but a fixture among the $\P$-positions of this family of games. Note that for \CN{3}{2} and \CN{5}{3}, the set of \P-positions equals $S_1$. These two games are too small to show the more general structure of the \P-positions of this family. The question arises whether there are generalizations of the other components of the \P-positions of \CN{7}{4} that play a part of the \P-positions in this family. The obvious candidate would be $S_2$, with all equal stack heights. Interestingly enough, this set is NOT a part of the \P-positions (except for the terminal position) of \CN{9}{5}. For example, the position $(2,2,2,2,2,2,2,2,2)$ is an \N-position of \CN{9}{5}.

\vspace{0.2in}

{\bf Acknowledgements} We would like to thank Kenneth A. Regas for the creation of the nice figures.

\bibliographystyle{unsrt}

\end{document}